\documentclass{amsart}
\usepackage{graphicx}
\usepackage{amsfonts}
\newtheorem{tm}{Theorem}

\newtheorem{rem}{Remark}
\newtheorem{rems}{Remarks}
\newtheorem{lm}{Lemma}
\newtheorem{ex}{Example}
\newtheorem{cor}{Corollary}
\newtheorem{prop}{Proposition}
\newtheorem{nota}{Notation}
\newtheorem{quest}{Question}
\newtheorem{cond}{Condition}

\begin{document}

\title{Descartes' rule of signs and moduli of roots}
\author{Vladimir Petrov Kostov}
\address{Universit\'e C\^ote d’Azur, CNRS, LJAD, France} 
\email{vladimir.kostov@unice.fr}

\begin{abstract}
  A hyperbolic polynomial (HP) is a real univariate polynomial with all roots
  real. By Descartes' rule of signs a HP with all coefficients nonvanishing
  has exactly $c$ positive and exactly $p$ negative roots counted with
  multiplicity, where $c$ and $p$ are the numbers of sign changes and sign
  preservations in the sequence of its coefficients. For $c=1$ and $2$, we
  discuss the question: When the
  moduli of all the roots of a HP are arranged in the increasing order on the
  real
  half-line, at which positions can be the moduli of its positive roots depending on the positions of the sign changes in the sequence of coefficients?\\ 


\end{abstract}
\maketitle 

\section{Introduction}

We consider {\em hyperbolic polynomials (HPs)}, i.e. real polynomials 
in one real variable with all roots real. We limit our study to the 
case when the polynomials are monic and all coefficients are nonvanishing.  
In this case Descartes' rule 
of signs implies that a degree $d$ HP has exactly $c$ positive and exactly $p$ 
negative roots (counted with multiplicity), 
where $c$ and $p$ are the numbers of sign changes and sign preservations 
in the sequence of coefficients of the polynomial (hence $c+p=d$). A 
{\em sign pattern (SP)} is a finite sequence of ``$+$'' and/or ``$-$''-signs 
beginning with a $+$. If a HP 
is denoted by $P:=x^d+\sum _{j=0}^{d-1}a_jx^j$, then we say that $P$ defines (or realizes) 
the SP (of length $d+1$) $(+$, sgn\,$(a_{d-1})$, sgn\,$(a_{d-2})$, $\ldots$, 
sgn\,$(a_0))$. It is true that: 
\vspace{1mm}

1) for every SP of length $d+1$, there exists a HP defining the given SP, see Remark~\ref{remconcat};

2) the all-pluses SP of length $d+1$ (hence with $c=0$) is realizable by any monic HP having $d$ negative roots.
\vspace{1mm}

Descartes' rule of signs does not impose any inequalities between the moduli 
of the positive and the negative roots of $P$. In the present paper we consider,  
for $c=1$ and $c=2$, the question: 

\begin{quest}\label{quest1}
When the moduli of all the roots of a HP are arranged in the increasing order on the real half-line, at which positions can be the moduli of the positive roots depending on the positions of the sign changes in the sequence of coefficients? In particular, at which positions can they be in the case when there are no equalities between moduli of roots? 
\end{quest}
To make formulations easier we fix the following notation:

\begin{nota}\label{notanota}
{\rm (1) For $c=1$, we denote by $\Sigma _{m,n}$ the SP consisting of $m$ pluses 
followed by $n$ minuses, where $1\leq m,n\leq d$, $m+n=d+1$. For $c=2$, we denote 
by $\Sigma _{m,n,q}$ the SP consisting of $m$ pluses followed by $n$ minuses 
followed by $q$ pluses, where $1\leq m,n,q\leq d-1$, $m+n+q=d+1$.

(2) For $c=1$, we denote by $0<\alpha$ the modulus of the positive root and by 
$0<\gamma _1\leq \cdots \leq \gamma _{d-1}$ the moduli of the negative roots 
of a degree $d$ HP. 
For $c=2$, we denote by $0<\beta \leq \alpha$ the moduli of its positive and by 
$0<\gamma _1\leq \cdots \leq \gamma _{d-2}$ the moduli of its negative roots. 
We set $\gamma :=(\gamma _1$, $\ldots$, $\gamma _{d-c})$.

(3) By $e_k(\gamma )$ we denote 
the $k$th elementary symmetric function of the quantities $\gamma _i$, i.e. 
$e_k(\gamma ):=\sum _{1\leq j_1<j_2<\cdots <j_k\leq d-c}\gamma _{j_1}\gamma _{j_2}
\cdots \gamma _{j_k}$, and by $e_k(\hat{\gamma _i})$ we denote this 
symmetric function of the quantities $\gamma _1$, $\ldots$, $\gamma _{i-1}$, 
$\gamma _{i+1}$, $\ldots$, $\gamma _{d-c}$.

(4) For $c=2$, we denote by $m^*$, $n^*$ and $q^*$ the numbers of moduli of negative roots of a HP defining this sign pattern which are respectively larger than $\alpha$, belonging to the interval $(\beta , \alpha )$, and smaller than $\beta$. In the absence of an equality $\gamma _j=\alpha$ or $\gamma _j=\beta$, one has  $m^*+n^*+q^*=d-2$. For $c=1$, $m^*$ (resp. $n^*$) stands for the number of moduli of negative roots which are larger (resp. smaller) than $\alpha$. In the absence of an equality $\gamma _j=\alpha$, one has $m^*+n^*=d-1$.}
\end{nota}

For $c=1$ and $2$, Question~\ref{quest1} can be formulated as follows: 

\begin{quest}\label{quest2}
For a given degree $d$, what can be the values of $m^*$ depending on $m$ (if $c=1$) or of $m^*$ and $n^*$ depending on $m$ and $n$ (if $c=2$)? Especially, what can these values be in the {\em generic} case when all moduli of roots are distinct?
\end{quest}

The answer to this question is not trivial. Thus the SP $\Sigma _{d,1}$ is realizable only by HPs with $m^*=d-1$ (hence $n^*=0$), see Theorem~\ref{tm1change}. In the cases of the SPs $\Sigma _{1,n,1}$ and $\Sigma _{m,1,q}$ one has respectively  $m^*=m-1$, $n^*=0$, $q^*=q-1$ (see Theorem~\ref{tmm1q}) and $m^*=q^*=0$, $n^*=d-2$ (see Theorem~\ref{tm2extrem}). In other situations there are several possibilities for these values, see Examples~\ref{exexex}, \ref{exd5} and \ref{exgreat} or Theorems~\ref{tm1change}, \ref{tmq1} and~\ref{tmq1bis}. 

\begin{rems}\label{remsaction}
{\rm (1) Replacing $P(x)$ by $(-1)^dP(-x)$ means exchanging $c$ with $p$ and 
changing the signs of all roots of $P$. Therefore when asking the question how 
the moduli of the positive and negative roots of $P$ can be ordered on the real 
positive half-line it suffices to consider the cases with $c\leq [d/2]$. In particular, to obtain the answer to this question for $d\leq 5$, it is sufficient to study the cases with $c=1$ and $c=2$.

(2) Replacing $P$ by its {\em reverted} polynomial $P^R(x):=x^dP(1/x)$ means 
changing all roots of $P$ by their reciprocals and 
reading backward the SP defined by~$P$. In particular, the SP $\Sigma _{m,n}$ becomes $\Sigma _{n,m}$ and the SP $\Sigma _{m,n,q}$ becomes $\Sigma _{q,n,m}$. In order to have again a monic polynomial one could replace the polynomial $P^R(x)$ by $P^R(x)/a_0$.

(3) For real, but not necessarily hyperbolic degree $d$ polynomials, one can ask the question:} 

\begin{quest}\label{quest3}
Given a SP with $c$ sign changes and $p$ sign preservations, for which pairs of nonzero integers $(pos, neg)$ satisfying the conditions $pos\leq c$, $neg\leq p$ and $c-pos\in 2\mathbb{N}\cup 0\ni p-neg$ do there exist such polynomials defining the given SP and having exactly $pos$ positive and $neg$ negative roots, all distinct?
\end{quest} 

{\rm It seems that the question has been explicitly formulated for the first time in \cite{AJS}. The answer to it is not trivial and the exhaustive one is known for $d\leq 8$, see \cite{Gr}, \cite{AlFu}, \cite{FoKoSh}, \cite{KoCzMJ} and~\cite{KoMB}. The proof of the realizability of certain cases is often done by means of a {\em concatenation lemma}, see Lemma~\ref{lmconcat} in Section~\ref{secgeneral}.

(4) A tropical analog of Descartes' rule of signs is proposed in \cite{FoNoSh}. Different aspects of metric inequalities involving moduli of roots of polynomials are addressed in \cite{AKNR} and~\cite{Fo}.}
\end{rems}

The paper is structured as follows. In Section~\ref{secex} we give examples of SPs and HPs realizing these SPs with given strict inequalities between the quantities $\alpha$, $\beta$ and $\gamma _j$. In Section~\ref{secc1} we consider the case $c=1$, i.e. the case of $\Sigma _{m,n}$, see Theorem~\ref{tm1change} and Corollary~\ref{cor1change} which provide the exhaustive answer to Question~\ref{quest2} in the generic case. The sections after Section~\ref{secc1} concern the situation when $c=2$. In Section~\ref{sec1n1} we consider the case $c=2$, $m=q=1$, $n=d-1$, i.e. the case of $\Sigma _{1,n,1}$, see Theorem~\ref{tm2extrem}. In Section~\ref{secq1} we consider the case $c=2$, $q=1$, i.e. the case of $\Sigma _{m,n,1}$, see Theorems~\ref{tmq1} and~\ref{tmq1bis}. In Section~\ref{secn1} we consider the case $n=1$, i.e. the one of $\Sigma _{m,1,q}$, $m+q=d$, see Theorem~\ref{tmm1q}. In Section~\ref{secgeneral} we formulate a concatenation lemma (Lemma~\ref{lmconcat}) which plays a key role in the construction of HPs realizing given SPs. With the help of this lemma we explain how for $c=2$, $n\geq 2$, one can prove the realizability of certain cases. We also sum up the realizability results of the present paper for HPs of degrees from $2$ to~$5$, with $c=2$.

\section{Examples\protect\label{secex}}

\begin{ex}\label{exd2}
{\rm (1) For $d=1$, there are two possible SPs, namely 
$(+,+)$ and $(+,-)=\Sigma _{1,1}$, realizable respectively by $x+1$ (with $\gamma _1=1$) and $x-1$ (with $\alpha =1$).
\vspace{2mm}

(2) For $d=2$, one has the SPs $(+,+,+)$, $(+,+,-)=\Sigma _{2,1}$, $(+,-,+)=\Sigma _{1,1,1}$ and $(+,-,-)=\Sigma _{1,2}$. They are realizable by the HPs 

$$\begin{array}{lllcllll}(x+1)(x+2)&=&x^2+3x+2&,&(x+2)(x-1)&=&x^2+x-2&,\\ \\ (x-1)(x-2)&=&x^2-3x+2&{\rm and}&(x+1)(x-2)&=&x^2-x-2&,\end{array}$$
with self-evident values of $\alpha$, $\beta$, $\gamma _1$ and $\gamma _2$. For any  HPs realizing the SPs $\Sigma _{2,1}$ or $\Sigma _{1,2}$, one has $\gamma _1>\alpha$ or $\gamma _1<\alpha$ respectively.}

\end{ex}

\begin{ex}\label{exd3}
{\rm (1) For $d=3$, we show SPs, HPs realizing them and inequalities between the moduli of their roots. The SP $\Sigma _{1,2,1}$ is realizable by the HPs}

$$\begin{array}{llllll}
P_1&:=&(x+1)(x-1.5)(x-1.6)&=&x^3-2.1x^2-0.7x+2.4&,\\ \\ 
P_2&:=&(x+1)(x-1.5)(x-0.6)&=&x^3-1.1x^2-1.2x+0.9&{\rm and}\\ \\ 
P_3&:=&(x+1)(x-0.5)(x-0.6)&=&x^3-0.1x^2-0.8x+0.3&.\end{array}$$
{\rm with $\gamma _1<\beta <\alpha$ or $\beta <\gamma _1<\alpha$ or $\beta <\alpha <\gamma _1$ respectively. 
\vspace{2mm}

(2) The SPs $\Sigma _{2,1,1}$ and $\Sigma _{3,1}$ are realizable by the HPs 

$$\begin{array}{llllll}
P_4&:=&(x+1)(x-0.2)(x-0.1)&=&x^3+0.7x^2-0.28x+0.02&{\rm and}\\ \\ 
P_5&:=&(x+1)(x+2)(x-0.1)&=&x^3+2.9x^2+1.7x-0.2&,\\ \\ 
\end{array}$$
with $\beta <\alpha <\gamma _1$ and $\alpha <\gamma _1<\gamma _2$ respectively. Hence the SPs $\Sigma _{1,1,2}$ and $\Sigma _{1,3}$ are realizable by the HPs 
$P_4^R$ and $-P_5^R$, with $\gamma _1<\beta <\alpha$ and $\gamma _1<\gamma _2<\alpha$ respectively, see part (2) of Remarks~\ref{remsaction}.
\vspace{2mm}

(3) The SP $\Sigma _{2,2}$ is realizable by the HPs 

$$\begin{array}{llllll}
P_6&:=&(x+1)(x+2)(x-0.95)&=&x^3+2.05x^2-0.85x-1.9&,\\ \\ 
P_7&:=&(x+1)(x+2)(x-1.5)&=&x^3+1.5x^2-2.5x-3&{\rm and}\\ \\  P_8&:=&(x+1)(x+2)(x-2.5)&=&x^3+0.5x^2-5.5x-5&,\end{array}
$$
with $\alpha =0.95<\gamma _1=1<\gamma _2=2$, with 
$\gamma _1=1<\alpha =1.5<\gamma _2=2$ or with $\gamma _1=1<\gamma _2=2<\alpha =2.5$ respectively.}
\end{ex}

\begin{ex}\label{exexex}
{\rm (1) For $d=4$, one has  

$$\begin{array}{l}Q_1:=(x-1.2)(x-0.8)(x+0.97)(x+0.98)\\ \\ =x^4-0.05x^3-1.9894x^2-0.0292x+0.912576\end{array}$$
with $\beta =0.8<\gamma _1=0.97<\gamma _2=0.98<\alpha =1.2$, so one realizes 
the SP $\Sigma _{1,3,1}$.  
\vspace{2mm}

(2) Again for $d=4$, one can realize the SP $\Sigma _{2,2,1}$ in different ways, with different inequalities between the quantities $\alpha$, $\beta$, $\gamma _1$ and $\gamma _2$. We list some examples here:  

$$Q_2:=(x-4)(x-1)(x+2.1)(x+3)=x^4+0.1x^3-15.2x^2-11.1x+25.2~,$$
i.e. for $\beta =1<\gamma _1=2.1<\gamma _2=3<\alpha =4$;  

$$\begin{array}{l}Q_3:=(x-0.995)(x-0.99)(x+1)(x+1.001)\\ \\ 
=x^4+0.016x^3-1.985935x^2-0.01589995x+0.98603505~,
\end{array}$$
i.e. for  
$\beta =0.99<\alpha =0.995<\gamma _1=1<\gamma _2=1.001$; 

$$Q_4:=(x-1.6)(x-1.5)(x+1)(x+100)=x^4+97.9x^3-210.7x^2-67.6x+240~,$$
i.e. for $\gamma _1=1<\beta =1.5<\alpha =1.6<\gamma _2=100$;

$$\begin{array}{l}Q_5:=(x-1)(x-0.97)(x+0.99)(x+1.001)\\ \\ =x^4+0.021x^3-1.96128x^2-0.0209803x+0.9612603~,\end{array}$$
i.e. for $\beta =0.97<\gamma _1=0.99<\alpha =1<\gamma _2=1.001$. When one replaces the latter four HPs by their reverted ones (see part (2) of Remarks~\ref{remsaction}), then one realizes the SP $\Sigma _{1,2,2}$, with $\alpha$ (resp. $\beta$ and $\gamma _j$) changed to $1/\beta$ (resp. $1/\alpha$ and $1/\gamma _{3-j}$).}

\end{ex} 

\begin{ex}\label{exd5}
{\rm For $d=5$, consider the SP $\Sigma _{2,2,2}$ and some HPs defining this SP:}

$$\begin{array}{ll}
(x-1)(x-1.05)(x+1.08)(x+1.09)(x+1.1)&\\ \\ 
=x^5+1.22x^4-2.0893x^3-2.57819x^2+1.087824x+1.359666&,\\ \\ 
(x-1)(x-1.05)(x+1.02)(x+1.09)(x+1.1)&\\ \\ 
=x^5+1.16x^4-2.0977x^3-2.443760x^2+1.097331x+1.284129\\ \\ 
(x-1)(x-1.05)(x+1.02)(x+1.04)(x+1.1)&\\ \\ 
=x^5+1.11x^4-2.1012x^3-2.33506x^2+1.101036x+1.225224&,\\ \\ 
(x-1)(x-1.05)(x+1.02)(x+1.03)(x+1.04)&\\ \\ 
=x^5+1.04x^4-2.1019x^3-2.187206x^2+1.1018508x+1.1472552&,\\ \\ 
(x-1)(x-1.05)(x+0.99)(x+1.09)(x+1.1)&\\ \\ 
=x^5+1.13x^4-2.1019x^3-2.376545x^2+1.1020845x+1.2463605&{\rm and}\\ \\ 
(x-1)(x-1.05)(x+0.99)(x+1.04)(x+1.1)&\\ \\ 
=x^5+1.08x^4-2.1039x^3-2.26927x^2+1.103982x+1.189188&.
\end{array}$$
{\rm It is easy to check that these HPs and their reverted ones realize all possible generic cases with this SP.}
\end{ex}

\begin{ex}\label{exgreat}
{\rm For $d=7$, the HP}

$$\begin{array}{l}
(x-1)(x+0.99)(x+0.94)(x+0.93)(x+0.92)(x+0.91)(x-0.9)\\ \\  =x^7+2.79x^6+0.7855x^5-4.244835x^4-3.88785176x^3+\\ \\ 0.8027291316x^2+2.102352335x+0.6521052938
  \end{array}$$
{\rm realizes the SP $\Sigma _{3,2,3}$, with} 

$$\beta =0.9<\gamma _1=0.91<\gamma _2=0.92<\gamma _3=0.93<\gamma _4=0.94<\gamma _5=0.99<\alpha =1~.$$
{\rm In this example one has $n^*=5$, $m^*=q^*=0$. More generally, consider the HP}

$$\begin{array}{l}(x^2-1)(x^2-0.9^2)(x+0.9)^3\\ \\  =x^7+2.7x^6+0.62x^5-4.158x^4-3.5883x^3+0.86751x^2+1.9683x+0.59049\end{array}$$
{\rm realizing the same SP. One can perturb its roots at $-1$ and $-0.9$ (the latter is $4$-fold) to obtain HPs with $n^*=0$, $1$, $2$, $3$, $4$ or~$5$ and with all moduli of roots distinct.}
\end{ex}

\section{The case $c=1$\protect\label{secc1}}

\begin{tm}\label{tm1change}
(1) Consider the SP $\Sigma _{m,n}$, where 
$1\leq n\leq m$. This SP is realizable by and  only by polynomials with $n^*\leq 2n-2$. 
In particular, for $n=1$, one has $m^*=d-1$, $n^*=0$.
\vspace{2mm}

(2) All cases described after 
the theorem are realizable.
\end{tm}

The cases in question are the ones when there are exactly $s$ quantities 
$\gamma _j$ 
which are equal to $\alpha$, exactly $r=n^*$ that are smaller than $\alpha$, 
where 
$s+r\leq 2n-2$, and exactly $d-1-s-r=m^*$ quantities $\gamma _j$ which are larger 
than 
$\alpha$. As for the quantities $\gamma _j$ which are smaller than $\alpha$, 
one can realize all 
possible cases of equalities and/or inequalities among them. When there are $<2n-2$ quantities 
$\gamma _j$ smaller than $\alpha$, the quantities $\gamma _j$ larger than $\alpha$ 
are presumed distinct. (However some more cases are realizable as well, 
see Remark~\ref{remmultvect}. Nothing is claimed about the cases 
which remain outside the reach of Remark~\ref{remmultvect}.) 
When there are exactly $2n-2$ quantities 
$\gamma _j$ smaller than $\alpha$, then among the quantities 
$\gamma _j$ larger than $\alpha$ one can have all 
possible equalities and/or inequalities. 

\begin{cor}\label{cor1change}
The SP $\Sigma _{m,n}$ with $1\leq m\leq n$ is realizable by and only by 
polynomials with $m^*\leq 2m-2$. In particular, for $n=d$, one has $m^*=0$, $n^*=d-1$.
\end{cor}

The corollary results from Theorem~\ref{tm1change}, see 
part (2) of Remarks~\ref{remsaction}. The realizable cases are easily 
deduced from the ones defined after Theorem~\ref{tm1change}. For $d=2$ and $d=3$, all generic  realizable cases covered by Theorem~\ref{tm1change} and Corollary~\ref{cor1change} are illustrated by Examples~\ref{exd2} and~\ref{exd3}.

\begin{rem}\label{rem1change}
 {\rm Suppose that one considers the question:} 
 \begin{quest}\label{quest4}
 For $d\geq 3$, given $n\in \mathbb{N}$, 
 $1\leq n\leq d$, what are the possible values of $n^*$?
 \end{quest} 
{\rm Theorem~\ref{tm1change} and Corollary~\ref{cor1change} imply that}

$$\max \, (0,2n-d-1)~\leq ~n^*~\leq ~\min \, (2n-2,d-1)~.$$
{\rm Indeed, from Theorem~\ref{tm1change} and Corollary~\ref{cor1change} one deduces the inequalities $n^*\leq 2n-2$ and $n^*\geq d-1-(2m-2)$, i.e. $n^*\leq 2n-2$ and $n^*\geq 2n-d-1$.}
\end{rem}

\begin{proof}[Proof of Theorem~\ref{tm1change}]
Suppose that $\gamma _j<\alpha$ for $j=1$, $\ldots$, $2n-1$. 
Set $\delta _j:=\gamma _j$, $j=1$, $\ldots$, $2n-1$, 
$\delta :=(\delta _1$, $\ldots$, $\delta _{2n-1})$ and 

$$Q:=(x-\alpha )\prod_{j=1}^{2n-1}(x+\delta _j)=
x^{2n}+a_{2n-1}x^{2n-1}+\cdots +a_1x+a_0~.$$ 
Hence $a_n=e_n(\delta )-\alpha e_{n-1}(\delta )$. Thus 

$$na_n=\sum _{i=1}^{2n-1}\delta _ie_{n-1}(\hat{\delta _i})-
\alpha \sum _{i=1}^{2n-1}e_{n-1}(\hat{\delta _i})=
\sum _{i=1}^{2n-1}(\delta _i-\alpha )e_{n-1}(\hat{\delta _i})<0~.$$
As $a_0=-\alpha \delta _1\cdots \delta _{2n-1}<0$ and as $P$ has one positive 
and $2n-1$ negative roots, one has exactly one sign change in the sequence 
$1$, $a_{2n-1}$, $\ldots$, $a_1$, $a_0$, so $a_j<0$ for $j\leq n$. 

Set $a_{-1}:=0$. The last $n+1$ coefficients of the polynomial 
$(x+\gamma _{2n})Q$ equal $a_{j-1}+\gamma _{2n}a_j<0$. In the same way the last 
$n+1$ coefficients of each of the polynomials 
$(\prod _{\nu =2n}^k(x+\gamma _{\nu}))Q$, $2n\leq k\leq d$, are negative which for $k=d$ 
leads to a contradiction with the definition of $\Sigma _{m,n}$.

To prove realizability of all cases mentioned after the lemma we observe first 
that for $R:=(x+1)^{2n-1}(x-1)=x^{2n}+g_{2n-1}x^{2n-1}+\cdots +g_1x+g_0$, 
one has $g_n=0$, $g_j>0$ for $j>n$ and $g_j<0$ for $j<n$. Consider for 
$\varepsilon >0$ small enough the polynomial 

$$\begin{array}{ccl}
\tilde{R}&:=&(x+1+\varepsilon u)^{2n-1-s-r}(x+1)^s(x+1-\varepsilon w)^r(x-1)\\ \\ 
&=&
x^{2n}+h_{2n-1}x^{2n-1}+\cdots +h_1x+h_0~,\end{array}$$
where $u>0$ and $w>0$; we set $\alpha :=1$. One has 

$$h_n=(C_{2n-2}^{n-1}-C_{2n-2}^n)((2n-1-s-r)u-rw)\varepsilon +o(\varepsilon )~,$$ 
with $C_{2n-2}^{n-1}-C_{2n-2}^n\neq 0$ and $2n-1-s-r\neq 0$, therefore one can choose $u$ and $w$ such that $h_n>0$ and $h_{n-1}<0$. 
After this 
one perturbs the quantities $\gamma _i$ which are smaller than $\alpha$ 
to obtain 
any possible case of equalities and/or inequalities among them by keeping the 
conditions $h_n>0$ and $h_{n-1}<0$. 
Then one sets 

$$K:=(1+\eta x)^{d-2n}\tilde{R}=x^d+\kappa _{d-1}x^{d-1}+\cdots +
\kappa _1x+\kappa _0~,$$ 
where $\eta >0$ is so small that $\kappa _n>0$ and $\kappa _{n-1}<0$. 
The polynomial $K$ has a $(d-2n)$-fold root $-1/\eta$ 
whose modulus is larger than $\alpha$. 

In the case when there are exactly $2n-2$ quantities $\gamma _j$ 
smaller than $\alpha$ one can perturb the $(d-2n)$-fold root $-1/\eta$ 
to obtain any possible case of equalities and inequalities among 
the $d-2n$ quantities $\gamma _j$ which are larger than $\alpha$. 
When there are less than $2n-2$ quantities $\gamma _j$ 
smaller than $\alpha$, not all quantities $\gamma _j$ larger than $\alpha$ 
can be obtained by perturbing $-1/\eta$. In this case one can make them all 
distinct by perturbing $-1/\eta$ and $-1-\varepsilon u$ into $d-2n$ and 
$2n-1-s-r$ distinct roots respectively.  

\end{proof}

\begin{rem}\label{remmultvect}
{\rm We call {\em multiplicity vector} 
a vector whose components are the multiplicities of the roots of a 
HP of a given degree; the roots are listed in the increasing order. 
Denote by $\vec{v}:=(\mu _1,\mu _2,\ldots ,\mu _k)$ 
the multiplicity vector of a degree $d-1-s-r$ HP. Hence $\mu _1+\cdots +\mu _k=d-1-s-r$. Suppose that $\vec{v}$ satisfies the following condition: 

\begin{cond}\label{condA} 
{\rm There exists an  
index $\nu$ such that $\mu _1+\cdots +\mu _{\nu}=d-2n$ hence 
$\mu _{\nu +1}+\cdots +\mu _k=2n-1-s-r$.} 
\end{cond}

The vector $\vec{v}$ can be viewed as the multiplicity vector of the roots of 
a polynomial which is obtained by perturbing the product $(x+1+\varepsilon u)^{2n-1-s-r}(1+\eta x)^{d-2n}$. When $\vec{v}$ satisfies Condition~\ref{condA}, the roots of $(x+1+\varepsilon u)^{2n-1-s-r}$ and the ones of $(1+\eta x)^{d-2n}$ can be perturbed independently. Thus when there are less than 
$2n-2$ quantities $\gamma _j$ 
smaller than $\alpha$, and when $\vec{v}$ satisfies Condition~\ref{condA}, 
one can realize the case of equalities and inequalities 
among the roots of the HP defined by the vector $\vec{v}$ by perturbing separately the roots 
$-1/\eta$ and $-1-\varepsilon u$. There remains to observe that for $\eta$ small enough, the root $-1/\eta$ is smaller than the root $-1-\varepsilon u$.}  
\end{rem}

\section{The case of $\Sigma _{1,n,1}$\protect\label{sec1n1}}

In the present section we consider SPs of the form $\Sigma _{1,n,1}$, i.e. with $c=2$, $m=q=1$ and $n=d-1$.

\begin{tm}\label{tm2extrem}
For $d\geq 4$, the SP $\Sigma _{1,d-1,1}$ is realizable by and 
only by HPs with $n^*=d-2$, $m^*=q^*=0$. 
\end{tm}

\begin{rem}
{\rm For $d=2$, no quantity $\gamma _j$ is defined, see Example~\ref{exd2}. For $d=3$, all possible cases of strict inequalities between the quantities $\alpha$, $\beta$ and $\gamma _1$ are realizable, see the HPs $P_1$, $P_2$ and $P_3$ in Example~\ref{exd3}, so Theorem~\ref{tm2extrem} does not hold true for $d=3$.}
\end{rem}

\begin{proof}
Consider a polynomial $Q:=x^d+a_{d-1}x^{d-1}+\cdots +a_0$ realizing the SP 
$\Sigma _{1,d-1,1}$. Hence $a_{d-1}<0$ and  $a_1<0$, i.e.

$$\begin{array}{rccccc}
-\alpha -\beta +\sum _{j=1}^{d-2}\gamma _j&=&a_{d-1}&<&0&{\rm and}\\ \\  
(\alpha \beta \prod_{j=1}^{d-2}(-\gamma _j))
(1/\alpha +1/\beta -\sum _{j=1}^{d-2}1/\gamma _j)&=&(-1)^{d-1}a_1&&&
\end{array}$$
which implies

\begin{equation}\label{Vieta}
\alpha +\beta >\sum _{j=1}^{d-2}\gamma _j~~~\, \, \, {\rm and}~~~\, \, \, 
1/\alpha +1/\beta >\sum _{j=1}^{d-2}1/\gamma _j~.
\end{equation}
If for at least two indices $j$ one has $\gamma _j\geq \alpha$ (resp. 
$\gamma _j\leq \beta$), then the first (resp. the second) of conditions 
(\ref{Vieta}) fails. The same holds true if there exist two indices $j_1$ and 
$j_2$ for which one has $\gamma _{j_1}\geq \alpha \geq \gamma _{j_2}\geq \beta$ 
(resp. $\alpha \geq \gamma _{j_1}\geq \beta \geq \gamma _{j_2}$). 
Thus for $d\geq 5$, the only 
possibility conditions (\ref{Vieta}) to hold true is to have 
$\beta <\gamma _j<\alpha$ for $j=1$, $\ldots$, $d-2$. 

For $d=4$, 
one has either $\alpha >\gamma _2\geq \gamma _1>\beta >0$ or  
$\gamma _2\geq \alpha \geq \beta \geq \gamma _1>0~(*)$. So to prove the theorem 
one has to refute possibility $(*)$. One can notice that it is 
impossible to have $\gamma _2=\alpha$ or $\beta = \gamma _1$ in which case 
at least one of conditions (\ref{Vieta}) fails. Therefore one has $\gamma _2-\gamma _1>\alpha -\beta ~(**)$.

Suppose that inequalities $(*)$ and (\ref{Vieta}) hold true. 
Then one can decrease 
continuously $\alpha$ until for the first time at least one of the three 
equalities holds true:

$$\alpha =\beta ~~~\, \, \, {\rm or}~~~\, \, \, 
\alpha +\beta =\gamma _1+\gamma _2~~~\, \, \, {\rm or}~~~\, \, \, 
1/\alpha +1/\beta =1/\gamma _1+1/\gamma _2~.$$
If this is $\alpha =\beta$, then $2\beta \geq \gamma _1+\gamma _2$ and 
$2/\beta \geq (\gamma _1+\gamma _2)/(\gamma _1\gamma _2)$, that is 

$$4\geq (\gamma _1+\gamma _2)^2/(\gamma _1\gamma _2)$$
which leads to $(\gamma _1-\gamma _2)^2\leq 0$. This is possible only if $\alpha =\beta =\gamma _1=\gamma _2$ which is a contradiction. If the equality is 
$\alpha +\beta =\gamma _1+\gamma _2$, then 

$$1/\alpha +1/\beta =(\gamma _1+\gamma _2)/(\alpha \beta )\geq 
(\gamma _1+\gamma _2)/(\gamma _1\gamma _2)$$
hence $\alpha \beta \leq \gamma _1\gamma _2$. Set $s:=(\alpha +\beta )/2$. Then 

$$\alpha =s+u~~~,~~~\beta =s-u~~~,~~~\gamma _1=s+v~~~,~~~\gamma _2=s-v~~~,~~~
0<u<v<s~.$$
(The inequality $u<v$ results from $(*)$ and $(**)$.) This implies the contradiction 
$\gamma _1\gamma _2=s^2-v^2<s^2-u^2=\alpha \beta$. 
Finally, if  
$1/\alpha +1/\beta =1/\gamma _1+1/\gamma _2=2t$, then 

$$1/\alpha =t-r~~~,~~~1/\beta =t+r~~~,~~~
1/\gamma _2=t-w~~~,~~~1/\gamma _1=t+w~~~,~~~0<r<w<t~,$$
hence $\alpha \beta =1/(t^2-r^2)<1/(t^2-w^2)=\gamma _1\gamma _2$. However one 
must have 

$$\alpha +\beta =2t/(t^2-r^2)>2t/(t^2-w^2)=\gamma _1+\gamma _2$$
which is a contradiction.

\end{proof}

\section{The case $q=1$\protect\label{secq1}}

Now we consider SPs of the form $\Sigma _{m,n,1}$, i.e. with $c=2$ and $q=1$. 

\begin{tm}\label{tmq1}
(1) For $d\geq 4$, a HP defining a SP $\Sigma _{m,n,1}$ satisfies one of the two conditions:
\vspace{2mm}

i) its root of smallest modulus is positive; 
\vspace{2mm}

ii) one has $\gamma _1\leq \beta \leq \alpha <\gamma _2\leq \cdots 
\leq \gamma _{d-2}$.
\vspace{2mm}

(2) If condition ii) is satisfied, then $n=2$ or $n=3$.
\vspace{2mm}

(3) For $n=3$ (resp. for $n=2$), and for any $d\geq 5$, there exist polynomials with roots satisfying conditions ii) for all possible choices of equalities or strict inequalities (resp. conditions ii) with all inequalities strict). 

\end{tm}

\begin{rem}
{\rm For $d=4$ and $n=3$, one deals with the SP $\Sigma _{1,3,1}$; this case is covered by Theorem~\ref{tm2extrem}. For $d=4$ and $n=2$, see the polynomials in part (2) of Example~\ref{exexex}; they correspond to all generic cases allowed by Theorem~\ref{tmq1}. For the case $n=q=1$ see Section~\ref{secn1}.}
\end{rem}

\begin{proof}
We denote a HP defining a SP $\Sigma _{m,n,1}$ by 
$T:=x^d+a_{d-1}x^{d-1}+\cdots +a_1x+a_0$. Recall that 

\begin{equation}\label{RRR}
1/\alpha +1/\beta -1/\gamma _1-\cdots -1/\gamma _{d-2}=-a_1/a_0>0
\end{equation}
(to see this it suffices to consider the polynomial 
$T^R(x):=x^dT(1/x)=a_0x^d+a_1x^{d-1}+\cdots +1$ whose roots are the reciprocals 
of the roots of $T(x)$). 
Hence at most one of the quantities 
$1/\gamma _j$ can be $\geq 1/\beta$ (so this is $1/\gamma _1$ and 
$\gamma _1\leq \beta$), 
otherwise inequality (\ref{RRR}) 
fails. If there exists exactly one such quantity, then 
for $j>1$, one has $\gamma _j>\alpha$. This proves part (1). 

Part (2). Suppose that condition {\em ii)} is satisfied. 
Consider the polynomial $T^R$ defined above. We denote by $1/\gamma$ 
the $(d-3)$-tuple $(1/\gamma _2,\ldots ,1/\gamma _{d-2})$ and by $e_j$ 
the quantity $e_j(1/\gamma )$. 
One has 

\begin{equation}\label{a4}\begin{array}{ccl}
a_4/a_0&=&e_4+(1/(\alpha \beta ))e_2+(1/\gamma _1)e_3+
(1/(\alpha \beta \gamma _1))e_1\\ \\ &&-(1/\alpha +1/\beta )e_3-
(1/\alpha +1/\beta )(1/\gamma _1)e_2~.\end{array}
\end{equation}
The following inequality holds true: 

\begin{equation}\label{e1e2}
e_2=((e_1)^2-\sum _{j=2}^{d-2}(1/\gamma _j)^2)/2
<(e_1)^2/2
\end{equation}
The inequalities (\ref{RRR}) and $1/\beta \leq 1/\gamma _1$ imply 
$1/\alpha >e_1$. 
Thus (see (\ref{e1e2})) $e_2<(e_1)^2/2<e_1/(2\alpha )$ which implies 

\begin{equation}\label{ineq1}
(1/\alpha +1/\beta )(1/\gamma _1)e_2<(2/(\beta \gamma _1))e_2<
(1/(\alpha \beta \gamma _1))e_1~.
\end{equation}
The inequality 

\begin{equation}\label{ineq2}
(1/\beta )e_3\leq (1/\gamma _1)e_3
\end{equation}
results from $1/\beta \leq 1/\gamma _1$ and the inequality

\begin{equation}\label{ineq3}
(1/\alpha )e_3\leq (1/(\alpha \beta ))e_2
\end{equation}
follows from $e_3<e_2e_1<(1/\alpha )e_2\leq (1/\beta )e_2$. Summing up 
inequalities 
(\ref{ineq1}), (\ref{ineq2}) and (\ref{ineq3}) one obtains $a_4/a_0>0$ (see (\ref{a4})) hence 
$a_4>0$ and $n\leq 3$.

There remains to exclude the case $n=1$. Suppose that the polynomial $T$ 
defines the SP $\Sigma _{d-1,1,1}$. Without loss of generality we assume 
that $\gamma _1 =1$ (this can be obtained by a linear change of the variable 
$x$). 
If $a_0>0$, $a_1<0$ and $a_2>0$, then 

$$\begin{array}{ll}
1/\alpha +1/\beta -1-e_1=-a_1/a_0>0&{\rm and}\\ \\ 
-1/\alpha -1/\beta +e_1+1/(\alpha \beta )-(1/\alpha +1/\beta )e_1+e_2=a_2/a_0>0&.
\end{array}$$
Set $\Delta :=1/\alpha +1/\beta -1$. Hence $\Delta >e_1>0$ and 

$$-1/\alpha -1/\beta +1/(\alpha \beta )>\Delta e_1-e_2>(e_1)^2-e_2>0$$
which by $\alpha \geq \beta \geq \gamma _1=1$ is impossible.

Part (3). For $n=3$, consider the polynomials $Y_s:=(x+s)^s(x-1)^2(x+1)$ 
(hence $d=s+3$, so $s\geq 2$). 
By Descartes' rule of signs there are exactly two sign changes in the sequence 
of coefficients of the polynomial $Y_s$. The last $5$ coefficients of $Y_s$ 
are the same 
as the ones of the polynomial 

$$W_s:=(C_s^4s^{s-4}x^4+C_s^3s^{s-3}x^3+C_s^2s^{s-2}x^2+
C_s^1s^{s-1}x+s^s)(x-1)^2(x+1)~,$$
where if $s-j<0$, then the term $C_s^js^{s-j}x^j$ is missing. 
These coefficients equal

$$\begin{array}{llll}
W_{s,0}=s^s&,&W_{s,1}=0&,\\ \\ W_{s,2}=-(1/2)(3s+1)s^{s-1}&,& 
W_{s,3}=-(1/3)(s-1)(s+1)s^{s-2}&\\ \\ {\rm and}&&W_{s,4}=
(1/8)(s+1)(3s^2+3s-2)s^{s-3}&.
\end{array}$$
For $s\geq 2$, one has $W_{s,2}<0$, $W_{s,3}<0$ and $W_{s,4}>0$. By an 
infinitesimal 
shift of the $s$-fold root at $(-s)$ one obtains the condition $W_{s,1}<0$. This is possible to do, because the coefficient of $x$ in the polynomial $(x+s+\varepsilon )^s(x-1)^2(x+1)$ equals $-s^{s-1}\varepsilon +o(\varepsilon )$. 
After this, 
if one wants to have strict inequalities instead of some of the equalities in 
the string 
of conditions {\em ii)} one can use infinitesimal shifting followed by 
bifurcation of 
the roots. 

For $n=2$, consider the polynomial $P_1$ of part (1) of Example~\ref{exd3}. 
For 
$\varepsilon >0$ small enough, the polynomial $(1+\varepsilon x)^{d-3}P_1$ 
defines 
the SP $\Sigma _{d-2,2,1}$ and has a $(d-3)$-fold root at $-1/\varepsilon$ and 
simple 
roots at $-1$, $1.5$, $1.6$. One can then perturb the root at $-1/\varepsilon$ 
to make all the roots of $(1+\varepsilon x)^{d-3}P_1$ distinct.

\end{proof}

In the following theorem we consider polynomials defining the SP 
$\Sigma _{m,n,1}$ with $m+n=d$ and satisfying the condition 
$\beta <\gamma _1$.

\begin{tm}\label{tmq1bis}
(1) If $m\leq n$, then there are $\leq 2m-1$ quantities $\gamma _j$ 
which are $\geq \alpha$ (that is, for $m<n-1$, one has 
$\gamma _{d-2m-1}<\alpha$).
\vspace{2mm}

(2) If $m\leq n$, then all cases when there are exactly $s\leq 2m-2$ quantities $\gamma _j$ not less than $\alpha$ are realizable by HPs.
\vspace{2mm}

(3) If $n<m$, then there are $\leq 2n-1$ quantities $\gamma _j$ 
which are $\leq \alpha$ (that is, one has 
$\gamma _{2n}>\alpha$).
\vspace{2mm}

(4) If $n<m$, then all cases when there are exactly $s\leq 2n-2$ 
quantities $\gamma _j$ not larger than $\alpha$ are realizable by HPs.
\end{tm}

\begin{proof}
Part (1). Suppose that a HP $P$ realizes the SP $\Sigma _{m,n,1}$. 
Hence its derivative is hyperbolic and realizes the SP 
$\Sigma _{m,n}$. 
Denote by $\alpha '$ and $\gamma _1'\leq \cdots \leq \gamma _{d-2}'$, 
the moduli of the latter's positive and negative roots. By   
Corollary~\ref{cor1change} at most $2m-2$ of the quantities $\gamma _j'$ are 
$\geq \alpha '$, i.e. the inequality $\gamma _{d-2m}'<\alpha '$ 
holds true (this inequality is meaningful only for $m<n$). For $j\geq 2$, one has 
$\gamma _{j-1}\leq \gamma _j'\leq \gamma _j$, so 
$\gamma _{d-2m-1}\leq \gamma _{d-2m}'<\alpha '$ (the left inequality
is meaningful only for $m<n-1$). On the other hand 
$\alpha '\leq \alpha$ which proves part (1) of the theorem.

Part (2). Denote by $Q$ a degree $d-1$ HP defining the SP $\Sigma _{m,n}$ and realizing the case $\gamma _{d-s-2}<\alpha \leq \gamma _{d-s-1}$; 
this case is defined without reference to $\beta$. To realize 
it is  
possible by Corollary~\ref{cor1change} and the lines that follow it. 
Set $P:=(x-\varepsilon )Q$, where $\varepsilon >0$ is small enough, 
so $P$ defines the SP $\Sigma _{m,n,1}$. (This statement is in fact a particular case of  Lemma~\ref{lmconcat}.)
Hence the root $\varepsilon$ of the polynomial $P$ is 
its root of smallest 
modulus and the polynomial $P$ realizes the case 

$$\beta <\gamma _1\leq \cdots \leq \gamma _{d-s-2}<
\alpha \leq \gamma _{d-s-1}\leq \cdots \leq \gamma _{d-2}~.$$

Part (3). Suppose that a HP $P$ realizes the SP $\Sigma _{m,n,1}$. 
Hence the reverted polynomial $P^R:=x^dP(1/x)$ is hyperbolic and 
realizes the SP $\Sigma _{1,n,m}$, and the polynomial 
$U:=dP^R-x(P^R)'$ realizes the SP $\Sigma _{n,m}$. Denote by 
$\alpha ^u$ and $\gamma _1^u\leq \cdots \leq \gamma _{d-2}^u$ the 
moduli of the latter's positive and negative roots. 
Hence $\alpha ^u\leq \alpha ^r$ (the superscript $r$ indicates 
moduli of roots of $P^R$) and by Corollary~\ref{cor1change}, 
$\gamma _{d-2n}^u<\alpha ^u$. The zeros of the polynomials $P^R$ 
and $U$ interlace, so $\gamma _{j-1}^r\leq \gamma _j^u\leq \gamma _j^r$. 
Thus $\gamma _{d-2n-1}^r\leq \gamma _{d-2n}^u<\alpha ^u\leq \alpha ^r$. 
The roots of $P^R$ are the reciprocals of the roots of $P$. Hence 
$\gamma ^r_j=1/\gamma _{d-1-j}$ and $\alpha ^r=1/\alpha$, therefore 
the inequality $\gamma _{d-2n-1}^r<\alpha ^r$ is equivalent to 
$1/\gamma _{2n}<1/\alpha$, i.e. to $\gamma _{2n}>\alpha$.

The proof of part (4) is done by analogy with the proof of 
part (2) -- one first finds a degree $d-1$ polynomial $Q$ 
defining the SP $\Sigma _{m,n}$ and 
realizing the case $\gamma _s\leq \alpha <\gamma _{s+1}$, and 
then constructs the polynomial
$P=(x-\varepsilon )Q$ which realizes the case 

$$\beta <\gamma _1\leq \cdots \leq \gamma _{s}\leq 
\alpha <\gamma _{s+1}\leq \cdots \leq \gamma _{d-2}~.$$
\end{proof}

\section{The case $n=1$\protect\label{secn1}}

We consider here SPs of the form $\Sigma _{m,1,q}$ (hence $c=2$, $n=1$ and $d=m+q$).

\begin{tm}\label{tmm1q}
The SP $\Sigma _{m,1,q}$ is realizable by and only by polynomials satisfying the condition 

\begin{equation}\label{ineqtm}
\gamma _1\leq \cdots \leq \gamma _{q-1}<\beta <\alpha <\gamma _q\leq \cdots \leq \gamma _{m+q-2}~,
\end{equation}
that is, with $m^*=m-1$, $n^*=0$ and $q^*=q-1$.
\end{tm}

\begin{rems}
{\rm (1) For $n=2$, unlike $n=1$, it is not true that there is a unique possibility for $m^*$ and $q^*$, see Examples~\ref{exexex}, \ref{exd5} and \ref{exgreat}. It would be interesting to know whether for $c=n=2$, there is an upper bound for the possible values of the quantity~$n^*$ (over all $m\geq 1$ and $q\geq 1$).

(2) The statement of part (1) of Theorem~\ref{tm1change} for $m=d$, $n=1$ (resp. the second sentence of Corollary~\ref{cor1change}) could be considered as an extension of the statement of Theorem~\ref{tmm1q} to the case $m=d$, $n=1$, $q=0$ (resp. $m=0$, $n=1$, $q=d$).}\end{rems}

\begin{proof}
$1^0$. We need the following lemma:

\begin{lm}\label{lmnotexist}
There exists no polynomial realizing the SP $\Sigma _{m,1,q}$ and 
satisfying the condition $\gamma _{\nu}=\alpha$ or $\gamma _{\nu}=\beta$ for some $\nu$ ($1\leq \nu \leq m+q-2$). 
\end{lm}

\begin{proof}
Suppose that such a polynomial $P:=\sum _{j=0}^da_jx^j$ exists. Then 
$P(\pm \gamma _{\nu})=0$ which implies $\sum _{k=0}^{[d/2]}a_{2k}(\gamma _{\nu})^{2k}=\sum _{k=0}^{[(d-1)/2]}a_{2k+1}(\gamma _{\nu})^{2k+1}=0$. This is impossible, because $\gamma _{\nu}>0$ and exactly one of the coefficients $a_j$ is negative while the rest are positive.
\end{proof}

$2^0$. For $m=n=q=1$, any hyperbolic degree $2$ polynomial has just two positive roots and there is nothing to prove. For $m=n=1$ and $q=2$, one has 

$$\alpha +\beta -\gamma _1>0~~~\, {\rm and}~~~\, \alpha \beta -(\alpha +\beta )\gamma _1>0~.$$
If $\gamma _1\geq \alpha$ (resp. if $\gamma _1\geq \beta$), then this leads to the contradiction $\alpha \beta /(\alpha +\beta )>\gamma _1>\alpha$, i.e. $\beta /(\alpha +\beta )>1$ (resp. $\alpha \beta /(\alpha +\beta )>\gamma _1>\beta$, i.e. $\beta /(\alpha +\beta )>1$). Hence $\gamma _1<\beta <\alpha$.

$3^0$. We perform induction on $q$ for $m$ fixed. We do this first for $m=1$ the induction base being the case $m=n=1$, $q=2$, see $2^0$; however the induction step is performed in the same way for any $m\geq 1$ fixed.

Suppose that a HP $P$ realizes the SP $\Sigma _{m,1,q}$ with $q>1$. Then its derivative $P'$ is a degree $d-1$ HP which realizes the SP $\Sigma _{m,1,q-1}$. Consider the family of polynomials $P_r:=rxP'+(1-r)P$, $r\in [0,1]$. For $r<1$, every polynomial of this family defines the SP $\Sigma _{m,1,q}$. Every polynomial of this family is hyperbolic. By Descartes' rule of signs every polynomial $P_r$ has exactly two positive roots and $d-2$ or $d-3$ negative ones (for $r\in [0,1)$ and $r=1$ respectively; for $r=1$, one of its roots equals $0$). For $r=1$, by inductive assumption, the moduli of the roots satisfy the inequalities 

$$0=\gamma _1<\gamma _2\leq \cdots \leq \gamma _{q-1}<\beta <\alpha <\gamma _q\leq \cdots \leq \gamma _{m+q-2}.$$
The roots of $P_r$ depend continuously on $r$, and for no value of $r\in [0,1]$ does one have an equality of the form 
$\gamma _{\nu}=\alpha$ or $\gamma _{\nu}=\beta$, see Lemma~\ref{lmnotexist}. Hence for 
$r\in [0,1)$, inequalities (\ref{ineqtm}) hold true. From our reasoning follows the proof of Theorem~\ref{tmm1q} for $m=1$. 

$4^0$. If a HP $P$ realizes the SP $\Sigma _{1,1,q}$ (hence $d=q+1$), then the HP 
$P^R:=x^dP(1/x)$ realizes the SP $\Sigma _{q,1,1}$ and $(P^R)'$ realizes the SP $\Sigma _{q,1}$. Hence the moduli $\alpha '$, $\gamma _1'$, $\ldots$, $\gamma _{d-2}'$ of the roots of the polynomial $(P^R)'$ satisfy the conditions 

$$\alpha '<\gamma _1'\leq \cdots \leq \gamma _{d-2}'~,$$
see Theorem~\ref{tm1change}. Consider the family of polynomials $(P^R)_r:=rx(P^R)'+(1-r)P^R$, $r\in [0,1]$. For $r\neq 1$ and close to $1$, the moduli of the roots of the polynomial $(P^R)_r$ satisfy the inequalities 

$$\beta <\alpha <\gamma _1\leq \cdots \leq \gamma _{d-2}~.$$
One can apply Lemma~\ref{lmnotexist} to conclude (by analogy with the reasoning about the family $P_r$ in $3^0$) that for $r\in [0,1)$, the above sequence of inequalities holds true and hence Theorem~\ref{tmm1q} holds true for any SP $\Sigma _{q,1,1}$ which we for convenience denote by $\Sigma _{m,1,1}$.  

$5^0$. Now one proves Theorem~\ref{tmm1q} by induction on $q$ for each $m$ fixed by applying the reasoning developed in $3^0$. The induction base are the cases $\Sigma _{m,1,1}$, see~$4^0$.

\end{proof}

\section{Comments on the case $c=2$\protect\label{secgeneral}}

\subsection{Concatenation lemma and its applications}

In the present subsection we consider the sign pattern $\Sigma _{m,n,q}$ with $n>1$ in the generic case. We remind that the quantities $m^*$, $n^*$ and $q^*$ are defined in Notation~\ref{notanota}.

We explain how using Theorem~\ref{tm1change} and Corollary~\ref{cor1change} one can prove the realizability of certain cases. To this end we recall a {\em concatenation lemma} proved in~\cite{FoKoSh}. We say that a real (not necessarily hyperbolic) polynomial realizes the pair $(pos, neg)$ if it has exactly $pos$ positive and exactly $neg$ negative roots counted with multiplicity. In what follows all real roots are presumed distinct.

\begin{lm}\label{lmconcat}
Suppose that the
monic polynomials $P_1$ and $P_2$ of degrees $d_1$ and $d_2$ with SPs $\sigma _1=(+,\hat{\sigma}_1)$ and $\sigma _2=(+,\hat{\sigma}_2)$, respectively, realize
the pairs $(pos_1, neg_1)$ and $(pos_2, neg_2)$. (Here $\hat{\sigma}_1$ and $\hat{\sigma}_2$ are
the SPs obtained from $\sigma _1$ and $\sigma _2$ by deleting the initial $+$-sign. Hence they can begin with $+$ or $-$.) Then

(1) if the last position of $\hat{\sigma}_1$ is $+$, then for any $\varepsilon >0$
small enough, the polynomial $\varepsilon ^{d_2}P_1(x)P_2(x/\varepsilon )$ realizes
the SP $(1,\hat{\sigma}_1,\hat{\sigma}_2)$ and the pair $(pos_1+pos_2, neg_1+neg_2)$;

(2) if the last position of $\hat{\sigma}_1$ is $-$, then for any $\varepsilon >0$
small enough, the polynomial $\varepsilon ^{d_2}P_1(x)P_2(x/\varepsilon )$ realizes
the SP $(1,\hat{\sigma}_1,-\hat{\sigma}_2)$ and the pair $(pos_1+pos_2, neg_1+neg_2)$. (Here $-\hat{\sigma}_2$ is the SP obtained
from $\hat{\sigma}_2$ by changing each $+$ by $-$ and vice versa.)
\end{lm}

\begin{rem}\label{remconcat}
{\rm One can prove that every SP $\sigma$ of length $d+1$, with $c$ sign changes and $p$ sign preservations, is realizable by a degree $d$ HP having $c$ distinct positive and $p$ distinct negative roots by applying $d-1$ times Lemma~\ref{lmconcat} with $P_2$ being each time a linear polynomial. If the second component of the SP $\sigma$ is $+$ (resp. is $-$), then one starts with $P_1=x+1$ (resp. with $P_1=x-1$). Suppose that one has thus constructed a degree $k$ HP $Q$, $1\leq k\leq d-1$, which realizes the SP $\sigma _k$ obtained from $\sigma$ by deleting the latter's last $d-k$ components. If the last two components of the SP $\sigma _{k+1}$ are different (resp. equal), then we apply Lemma~\ref{lmconcat} with $P_1=Q$ and $P_2=x-1$ (resp. with $P_1=Q$ and $P_2=x+1$). In this way the number of sign changes (resp. of sign preservations) of $\sigma _{k+1}$ is equal to the number of positive (resp. of negative) roots of the HP which realizes it. When one applies successively Lemma~\ref{lmconcat}, each next root (this is the root of $P_2(x/\varepsilon )$) has a modulus much smaller than the least of the moduli of the roots of $P_1$; this follows from the necessity to choose at each concatenation the number $\varepsilon$ sufficiently small. Therefore the moduli of the roots of the thus constructed HP realizing the SP $\sigma$ are all distinct. Moreover, the decreasing order of the moduli of positive and negative roots on $\mathbb{R}_+$ is the same as the order of sign changes and sign preservations when the SP is read from left to right. We call this order {\em canonical}. Thus it is always possible to realize a SP by a HP with canonical order of the moduli of its positive and negative roots on $\mathbb{R}_+$. The SPs $\Sigma _{1,d}$, $\Sigma _{d,1}$, $\Sigma _{1,d-1,1}$ and $\Sigma _{m,1,q}$ have only canonical realizations, see Theorem~\ref{tm1change}, Corollary~\ref{cor1change}, Theorem~\ref{tm2extrem} and Theorem~\ref{tmm1q}. For some SPs, not canonical realizations are also possible, see Examples~\ref{exexex}, \ref{exd5} and \ref{exgreat} or Theorems~\ref{tm1change}, \ref{tmq1} and~\ref{tmq1bis}.}  
\end{rem}

Now we explain how Lemma~\ref{lmconcat} can be used to construct real polynomials defining a given SP and realizing a given pair $(pos, neg)$. We are interested mainly in the case of HPs. Suppose that the polynomials $P_1$ and $P_2$, of degrees $m+n^{\flat}-1$ and $n^{\sharp}+q-1$, define the 
SPs $\Sigma _{m,n^{\flat}}$ and $\Sigma _{n^{\sharp},q}$, where $n^{\flat}+n^{\sharp}=n+1$. Then the polynomial 
$\varepsilon ^{d_2}P_1(x)P_2(x/\varepsilon )$ realizes the SP $\Sigma _{m,n,q}$, see part (2) of Lemma~\ref{lmconcat}.
Suppose that: 
\vspace{1mm}

{\em i)} exactly $m^*$ moduli of negative roots of $P_1$ are larger than the modulus of its positive root $\alpha$ and hence exactly $m+n^{\flat}-2-m^*$ such moduli are smaller than $\alpha$; 

{\em ii)} exactly $q^*$ moduli of negative roots of $P_2$ are smaller than the modulus of its positive root $\beta$ and hence exactly $n^{\sharp}+q-2-q^*$ such moduli are larger than $\beta$. 
\vspace{1mm}

For $\varepsilon >0$ small enough, the moduli of all roots of $P_2(x/\varepsilon )$ are smaller than the modulus of any of the roots of $P_1$.  Therefore the polynomial $\varepsilon ^{d_2}P_1(x)P_2(x/\varepsilon )$ has exactly $m^*$ moduli of negative roots which are larger than $\alpha$, exactly 

$$m+n^{\flat}-2-m^*+n^{\sharp}+q-2-q^*=m+n+q-3-m^*-q^*=d-2-m^*-q^*=n^*$$
such moduli belonging to the interval $(\varepsilon \beta , \alpha)$, and exactly $q^*$ such moduli which are smaller than $\varepsilon \beta$. The possible values of $m^*$, $n^{\flat}$,$n^{\sharp}$ and $q^*$ can be deduced from Theorem~\ref{tm1change} and Corollary~\ref{cor1change}.

\subsection{The case $2\leq d\leq 5$, $c=2$}

We sum up here what is proved in this paper about realizability of SPs in the generic case for $2\leq d\leq 5$, $c=2$. We remind that for given $d$, knowing the exhaustive answer to the question about  realizability of SPs with $c$ sign changes implies knowing the one for SPs with $d-c$ sign changes as well, see part (1) of Remarks~\ref{remsaction}. For $c=0$ (hence for $c=d$), the exhaustive answer is given by the observation 2) at the beginning of this paper. For $c=1$ (hence for $c=d-1$) the answer is given by Theorem~\ref{tm1change} and Corollary~\ref{cor1change}. For $2\leq d\leq 5$, we present here the exhaustive answer for $c=2$ (hence for $c=d-2$ as well). Thus for $d\leq 5$, we cover all possible generic cases.

For $d=2$, the exhaustive answer is provided by Example~\ref{exd2}. 

For $d=3$, the polynomials $P_1$, $P_2$ and $P_3$ (see Example~\ref{exd3}) show that the SP $\Sigma _{1,2,1}$ is realizable in all three possible generic situations of inequalities between the quantities $\alpha$, $\beta$ and $\gamma _1$. The SPs $\Sigma _{2,1,1}$ and $\Sigma _{1,1,2}$ have only canonical realizations, see Theorem~\ref{tmm1q} and Remark~\ref{remconcat}; examples of such realizations are given by the polynomials $P_4$ and  $P_4^R$, see Example~\ref{exd3}.

For $d=4$, the SPs $\Sigma _{1,1,3}$, $\Sigma _{2,1,2}$, $\Sigma _{3,1,1}$ and $\Sigma _{1,3,1}$ have only canonical realizations, see Theorems~\ref{tmm1q} and \ref{tm2extrem}. The realizable generic cases for the SP $\Sigma _{2,2,1}$ are illustrated in part (2) of Example~\ref{exexex}.  The cases $\gamma _1<\gamma _2<\beta <\alpha$ and $\gamma _1<\beta <\gamma _2<\alpha$ are not realizable with this SP, see Theorem~\ref{tmq1}. The corresponding results about the SP $\Sigma _{1,2,2}$ are then deduced using part (2) of Remarks~\ref{remsaction}.

For $d=5$, the SPs $\Sigma _{1,1,4}$, $\Sigma _{2,1,3}$, $\Sigma _{3,1,2}$, $\Sigma _{4,1,1}$ and $\Sigma _{1,4,1}$ have only canonical realizations, see Theorems~\ref{tmm1q} and \ref{tm2extrem}. The SP $\Sigma _{2,2,2}$ is realizable in all generic cases, see Example~\ref{exd5}. 
Consider the HP

\begin{equation}\label{eqexample}(x-0.1)(x-1)(x+1)^3=x^5+1.9x^4-0.2x^3-2x^2-0.8x+0.1~.\end{equation}
It defines the SP $\Sigma _{2,3,1}$ and one has $\beta =0.1$, $\alpha =\gamma _j=1$, $j=1$, $2$ and $3$. When its triple root at $-1$ bifurcates into three simple negative roots, its coefficients depend continuously on the bifurcation, so by nearby HPs one can realize the generic cases 

$$\begin{array}{cccc}
\beta <\alpha <\gamma _1<\gamma _2<\gamma _3&,&\beta <\gamma _1<\alpha <\gamma _2<\gamma _3&,\\  \\ \beta <\gamma _1<\gamma _2<\alpha <\gamma _3&{\rm and}&\beta <\gamma _1<\gamma _2<\gamma _3<\alpha&.
\end{array}$$
The SP $\Sigma _{2,3,1}$ is realizable by the HP

$$(x+1)(x-1)^2(x+2.1)^2=x^5+3.2x^4-0.79x^3-7.61x^2-0.21x+4.41~.$$
For $\varepsilon >0$ small enough, this is also the case of the polynomial 

$$(x+1)(x-1-\varepsilon )(x-1-2\varepsilon )(x+2.1+\varepsilon )(x+2.1+2\varepsilon )~~~\, ,~~~\, {\rm where}$$
$$\gamma _1=1<\beta =1+\varepsilon <\alpha =1 +2\varepsilon <\gamma _2=2.1+\varepsilon <\gamma _3=2.1+2\varepsilon ~.$$
According to Theorem~\ref{tmq1} there are no other realizable generic cases with the SP $\Sigma _{2,3,1}$. Taking the reverted of the HPs realizing the SP $\Sigma _{2,3,1}$ one realizes the SP $\Sigma _{1,3,2}$ in the corresponding generic cases, see part (2) of Remarks~\ref{remsaction}.

The SP $\Sigma _{3,2,1}$ is not realizable in the generic cases with $\gamma _1<\gamma _2<\beta$ or with $\gamma _1<\beta <\gamma _2<\alpha$, see Theorem~\ref{tmq1}. It is realizable for $\gamma _1<\beta <\alpha <\gamma _2<\gamma _3$ by the HP

$$\begin{array}{l}
(x+1)(x-1.5)(x-1.6)(x+100)(x+1000)=\\ \\ 
x^5+1097.9x^4+97689.3x^3-2.107676\times 10^5x^2-67360x+2.4\times10^5~.
\end{array}$$
In the generic cases $\beta <\alpha <\gamma _1<\gamma _2<\gamma _3$, $\beta <\gamma _1<\alpha <\gamma _2<\gamma _3$ and $\beta <\gamma _1<\gamma _2<\alpha <\gamma _3$ the SP $\Sigma _{3,2,1}$ is realizable by HPs of the form $(1+\varepsilon x)Q_3(x)$, $(1+\varepsilon x)Q_4(x)$ and $(1+\varepsilon x)Q_2(x)$ respectively, where $\varepsilon >0$ is small enough and the HPs $Q_j$ are the ones from Example~\ref{exexex}. Indeed, the leading coefficient in all three cases equals $\varepsilon >0$ and the other coefficients are close to the ones of $Q_j$. The quantity $\gamma _3$ equals $1/\varepsilon$. 

\begin{prop}\label{propprop}
The SP $\Sigma _{3,2,1}$ is not realizable in the generic case $\beta <\gamma _1<\gamma _2<\gamma _3<\alpha$.
\end{prop}

Before proving the proposition we remind that making use of part (2) of Remarks~\ref{remsaction} and knowing the answer about the SP $\Sigma _{3,2,1}$ one obtains the answer to the question in which generic cases is realizable the SP~$\Sigma _{1,2,3}$.

\begin{proof}
$1^0$. We consider the polynomial $P:=x^5+a_4a^4+a_3x^3+a_2x^2+a_1x+a_0$. We denote by $\tilde{R}\cong \mathbb{R}^5$ the space of the coefficients $a_j$. We consider the set 

$$\mathbb{R}_+^5~\supset ~U~:=~\{ ~\Gamma :=(\alpha , \beta , \gamma _1, \gamma _2, \gamma _3)~|~\beta <\gamma _1<\gamma _2<\gamma _3<\alpha ~\} $$
and its image $V$ in $\tilde{R}$ via the Vietta mapping which sends the quintuple $\Gamma$ into the quintuple of coefficients of the polynomial $(x-\alpha )(x-\beta )(x+\gamma _1)(x+\gamma _2)(x+\gamma _3)$ (excluding the leading coefficient $1$).  The closure $\bar{U}$ consists of $U$ and of quintuples $\Gamma$ for which at least one of the following equalities holds true:

\begin{equation}\label{eq5cases}
\beta =0~~~,~~~\beta =\gamma _1~~~,~~~\gamma _1=\gamma _2~~~,~~~\gamma _2=\gamma _3~~~,~~~\gamma _3=\alpha~.
\end{equation}
\begin{lm}\label{lmexclude}
There exists no HP defining the SP $\Sigma _{3,2,1}$ and satisfying at least one of the equalities~(\ref{eq5cases}).
\end{lm}

The lemma is proved after the proposition.
Thus if some HP defined by a quintuple $\Gamma _0$ defines the SP $\Sigma _{3,2,1}$, then $\Gamma _0$ is from the interior of $U$. The set $U$ being contractible one can connect $\Gamma _0$ by a $C^1$-smooth path $\mathcal{P}\subset U$ with a quintuple $\Gamma _1$ from the interior of $U$ which realizes the SP $\Sigma _{2,3,1}$; as we saw in the lines that follow equality~(\ref{eqexample}), such a quintuple $\Gamma _1$ exists. 

The path $\mathcal{P}$ intersects at least one of the hyperplanes $\{ a_j=0\} \subset \tilde{R}$. 

\begin{lm}\label{lmVietta}
The Vietta mapping is a local diffeomorphism at any point of the interior of $U$ onto its image.
\end{lm}

The lemma is proved at the end of the paper. It implies that one can modify the path $\mathcal{P}$ so that at any point of $\mathcal{P}$ at most one equality of the form $a_j=0$ holds true. Moreover, one can parametrize $\mathcal{P}$ by $t\in [0,1]$ so that for any point satisfying the equality $a_j=0$ there exists an open interval $(u,v)=J\subset [0,1]$ such that 

{\em i)} $a_j=0$ for $t=(u+v)/2$, 

{\em ii)} $a_j\neq 0$ for $t\in J\setminus \{ (u+v)/2\}$, 

{\em iii)} $a_j$ has different signs for $t\in (u,(u+v)/2)$ and $t\in ((u+v)/2,v)$ and 

{\em iv)} for $t\in J$, there exists a single index $j$ with $a_j$ satisfying properties {\em i)} -- {\em iii)}. 

Consider the point $\Gamma ^*\in \mathcal{P}$ closest to $\Gamma _0$ for which one has $a_j=0$ for some $j$. One cannot have $j=0$, because $\beta >0$. It is impossible to have $j=1$ or $j=2$, because then for $t\in ((u+v)/2,v)$, one would have the SP $\Sigma _{3,1,2}$ or $\Sigma _{4,1,1}$ realized by a quintuple from $U$ which contradicts Theorem~\ref{tmm1q}. One cannot have $j=4$ either, because then for $t\in ((u+v)/2,v)$, one would have $c=4$. There remains the only possibility $j=3$. 

\begin{lm}\label{lmexcludebis}
There exists no degree $5$ HP satisfying the conditions 
$0<\beta <\gamma _1<\gamma _2<\gamma _3<\alpha$ and $a_4>0$, $a_3=0$, $a_2<0$, $a_1<0$, $a_0>0$.
\end{lm}

The lemma (whose proof follows) finishes the proof of Proposition~\ref{propprop}. Indeed, by Lemma~\ref{lmexclude} no point of the boundary of $U$ realizes the SP $\Sigma _{3,2,1}$, and we just showed that this cannot be the case of a point $\Gamma _0$ from the interior of $U$ either. 
\end{proof}

\begin{proof}[Proof of Lemma~\ref{lmexcludebis}]
Suppose that such a HP exists. Recall that we denote by $e_1$ and $e_2$ the quantities $\gamma _1+\gamma _2+\gamma _3$ and $\gamma _1\gamma _2+\gamma _1\gamma _3+\gamma _2\gamma _3$. Then $a_4=-\alpha -\beta +e_1>0$, i.e. $0<\beta <e_1-\alpha$, and 

$$a_3=\alpha \beta -(\alpha +\beta )e_1+e_2=0~~~{\rm ,~\. i.~e.}~~~
\beta =(\alpha e_1-e_2)/(\alpha -e_1)~.$$
But $\alpha e_1-e_2=(\alpha -\gamma _1)\gamma _2+(\alpha -\gamma _2)\gamma _3+(\alpha -\gamma _3)\gamma _1>0$ while $\alpha -e_1<0$. Hence $\beta <0$ -- a contradiction.
\end{proof}

\begin{proof}[Proof of Lemma~\ref{lmexclude}]
Suppose that such a HP $T$ exists. Then $\beta >0$, otherwise $T(0)=0$ and $T$ does not define the SP $\Sigma _{3,2,1}$. Hence $\alpha >0$ and $\gamma _j>0$, $j=1$, $2$ and $3$. 

Suppose that $\beta =\gamma _1$. Set 

$$F:=(x+\gamma _2)(x+\gamma _3)(x-\alpha )=x^3+Ax^2+Bx+C~.$$
Then $B=-(\gamma _2+\gamma _3)\alpha +\gamma _2\gamma _3$ and 

$$T=(x^2-\beta ^2)F=x^5+Ax^4+(B-\beta ^2)x^3+(C-\beta ^2A)x^2-\beta ^2Bx-\beta ^2C~.$$
The condition $B-\beta ^2>0$ implies 

$$(\gamma _2+\gamma _3)\alpha <\gamma _2\gamma _3-\beta ^2<\gamma _2\gamma _3~.$$ 
However $(\gamma _2+\gamma _3)\alpha >\gamma _2\alpha \geq \gamma _2\gamma _3$ which is a contradiction. 

Suppose that $\alpha =\gamma _3$. Set 

$$G:=(x+\gamma _1)(x+\gamma _2)(x-\beta )=x^3+A^*x^2+B^*x+C^*~.$$
Then $B^*=-(\gamma _1+\gamma _2)\beta +\gamma _1\gamma _2$ and

$$T=(x^2-\alpha ^2)G=x^5+A^*x^4+(B^*-\alpha ^2)x^3+(C^*-\alpha ^2A^*)x^2-\alpha ^2B^*x-\alpha ^2C^*~.$$
On the one hand one must have $B^*-\alpha ^2>0$, but on the other  

$$B^*-\alpha ^2=-(\gamma _1+\gamma _2)\beta +(\gamma _1\gamma _2-\alpha ^2)<0$$
which is a contradiction. Suppose that $\gamma _j=\gamma _{j+1}=g$, where $j=1$ or $2$. We set 

$$T:=(x-\alpha )(x-\beta )(x+g)^2(x+h)=x^5+Mx^4+Nx^3+\cdots ~,$$
where $h=\gamma _3$ if $j=1$ and $h=\gamma _1$ if $j=2$. Then 

$$M=-\alpha -\beta +2g+h~~~\, {\rm and}~~~\, N=\alpha \beta -2g\alpha -2g\beta +g^2-h\alpha -h\beta +2gh~.$$
Hence $N^0:=N+\beta M=-2g\alpha +g^2-g\alpha +2gh-\beta ^2>0$. But for $j=1$, one has $g\leq h$ and 

$$N^0=(-2g\alpha +2gh)+(g^2-h\alpha )-\beta ^2<0$$
while for $j=2$, one has $h\leq g$ and 

$$N^0=(-g\alpha +gh)+(-h\alpha +gh)+(-g\alpha +g^2)-\beta ^2<0$$
which is a contradiction.

\end{proof}

\begin{proof}[Proof of Lemma~\ref{lmVietta}]
Consider the Vandermonde mapping 

$$(\beta , \gamma _1, \gamma _2, \gamma _3, \alpha )\mapsto (\varphi _1, \varphi _2, \varphi _3, \varphi _4, \varphi _5)~,$$
where $\varphi _k=\beta ^{5-k}+(-\gamma _1)^{5-k}+(-\gamma _2)^{5-k}+(-\gamma _3)^{5-k}+\alpha ^{5-k}$. For each point of the interior of $U$, this mapping defines a local diffeomorphism, because its determinant is up to a constant nonzero factor the Vandermonde determinant 
$W(\beta , \gamma _1, \gamma _2, \gamma _3, \alpha )\neq 0$. On the other hand the quantities $\varphi _k$ and $a_k$ are connected with one another by formulas of the form 

$$(5-k)a_k=-5\varphi _k+Q_k(\varphi _4,\ldots ,\varphi _{k+1})~~~,~~~5\varphi _k=-(5-k)a_k+Q_k^*(a_4,\ldots ,a_{k+1})~,$$
where $Q_k$ and $Q_k^*$ are polynomials. Hence the Vietta mapping also defines a local diffeomorphism.
\end{proof}


\begin{thebibliography}{Dillo 83}

\bibitem{AlFu} A.~Albouy, Y.~Fu: Some remarks about Descartes' rule 
of signs. Elem. Math., 69 (2014), 186--194. Zbl 1342.12002, 
MR3272179

\bibitem{AJS} B.~Anderson, J.~Jackson and M.~Sitharam: Descartes’ rule of 
signs revisited. Am. Math. Mon. 105 (1998), 447-- 451. 
Zbl 0913.12001, MR1622513

\bibitem{AKNR}M.~Avendano, R.~Kogan, M.~Nisse and J. Maurice Rojas, Metric estimates and membership complexity for archimedean amoebae and tropical hypersurfaces,  	arXiv:1307.3681 [math.AG].


\bibitem{Fo}J.~Forsg\aa rd, On the multivariate Fujiwara bound for exponential sums,  arXiv:1612.03738 [math.AG].

\bibitem{FoKoSh} J.~Forsg\aa rd, V.~P.~Kostov and B.~Shapiro: 
Could Ren\'e Descartes have known this?  Exp. Math. 24 (4)  
(2015), 
438-448. Zbl 1326.26027, MR3383475

\bibitem{FoNoSh}  J.~Forsg\aa rd, D.~Novikov and  B.~Shapiro, A tropical analog of Descartes' rule of signs, arXiv:1510.03257 [math.CA].

\bibitem{Gr} D.~J.~Grabiner: Descartes’ Rule of Signs: 
Another Construction. Am. Math. Mon. 106 (1999), 854--856.
Zbl 0980.12001, MR1732666



\bibitem{KoCzMJ}V.~P.~Kostov, On realizability of sign patterns by real polynomials, 
Czechoslovak Math. J. 68 (143) (2018), no. 3, 853–874.

\bibitem{KoMB}V.~P.~Kostov, Polynomials, sign patterns and Descartes' rule of signs, 
Mathematica Bohemica 144 (2019), No. 1, 39-67.





\end{thebibliography}
\end{document}